\documentclass[a4paper, oneside, reqno, 12pt]{amsart}
\usepackage[english]{babel}
\usepackage[utf8]{inputenc}
\usepackage{amsmath, amssymb, amsfonts}
\usepackage{mathtools}
\usepackage{fullpage}
\usepackage{comment}
\usepackage{textcomp, cmap}

\usepackage[usenames,dvipsnames]{xcolor}

\usepackage[backref=page]{hyperref}
\hypersetup{
	colorlinks   = true, 
	urlcolor     = blue, 
	linkcolor    = Blue, 
	citecolor   = Green 
}
\usepackage[capitalise, nameinlink, noabbrev]{cleveref}
\usepackage{autonum}

\usepackage{pgf, tikz}
\usetikzlibrary{arrows}

\usepackage{amsthm}
\theoremstyle{plain}
\newtheorem{theorem}{Theorem}

\newtheorem{lemma}[theorem]{Lemma}

\theoremstyle{definition}

\theoremstyle{remark}
\newtheorem*{remark*}{Remark}

\DeclareMathOperator{\conv}{conv}

\DeclareMathOperator{\tr}{tr}
\DeclareMathOperator{\rank}{rank}

\title[Perron and Frobenius Meet Carath\'{e}odory]{Perron and Frobenius Meet Carath\'{e}odory}
\author[M\'{a}rton Nasz\'{o}di, Alexandr Polyanskii]{M\'{a}rton Nasz\'{o}di, Alexandr Polyanskii}%

\address{M\'{a}rton Nasz\'{o}di, 
\newline\hphantom{iii} \href{https://dcg.epfl.ch}{DCG EPFL}; \href{https://www.math.elte.hu/en/}{Inst. of Math., ELTE}
}
\address{Alexandr Polyanskii,
\newline\hphantom{iii} \href{https://mipt.ru/english/}{MIPT}; \href{http://cmcagu.ru/}{CMC ASU}; \href{http://iitp.ru/en/about}{IITP RAS}
}
\begin{document}

\thispagestyle{empty}

\begin{abstract}
We present a new approach of proving certain Carath\'{e}odory-type theorems using the Perron--Frobenius Theorem, a classical result in matrix theory describing the largest eigenvalue of a matrix with positive entries.

One of the problems left open in this note is whether our approach may be extended to prove similar  results in the area, in particular the Colourful Carath\'{e}odory Theorem.
\end{abstract}
\maketitle

\section{Introduction}
\label{section introduction}
Carath\'{e}odory's Theorem~\cite{caratheodory1907} is one of conerstones of combinatorial geometry as Eckhoff~\cite{eckhoff1993helly} called it. The theorem claims that a point in the convex hull of a set $P\subseteq \mathbb R^d$ is in the convex hull of at most $d+1$ points of $P$; see the definition of convex hull in \cref{section preliminaries}.
The survey \cite{eckhoff1993helly} contains a number of generalizations of the theorem. As it usually happens with fundamental statements, Carath\'{e}odory's Theorem is closely connected with many other classical results in convex geometry such as Helly's and Radon's Theorems~\cite{helly1923mengen,radon1921mengen}. In this note, we provide a new proof of the theorem that shows a connections with a principal theorem in matrix theory, the Perron--Frobenius Theorem. 

Our proofs use induction on the number of points (but not on the dimension). Two, most likely related, questions remain open. First, can induction be avoided in the proofs? Second, does our method yield a proof of other Carath\'{e}odory-type theorems, specifically, B\'{a}r\'{a}ny's Colourful Carath\'{e}odory Theorem \cite{barany1982generalization}.

The paper is organized as follows. In \cref{section preliminaries}, we introduce notations and the tools that we use. In \cref{section proofs}, we illustrate the key idea applying it to prove Rankin's Theorem~\cite[Theorem~1(iii, iv)]{rankin1955}. In \cref{section main}, we use the idea to prove Carath\'{e}odory's and Steinitz's Theorems~\cite{caratheodory1907, steinitz1913}.

\section{Preliminaries}
\label{section preliminaries}
\subsection{Notations}
We write $[n]=\{1,\dots,n\}$, for a positive integer $n$. 
The \textit{convex hull}\ of a finite set of points $\{x_1,\dots,x_n\} \subset \mathbb R^d$, written $\conv \{x_1,\dots,x_n\}$, is the set \[
\left\{\sum_{i=1}^n\lambda_i x_i: \lambda_i\ge 0 \text{ for all }i\in[n], \sum_{i=1}^n\lambda_i=1\right\}.
\]
A point $x$ of a set $P\subseteq \mathbb R^d$ is an \textit{interior point} of $P$, if the set $P$ contains some open ball with center in $x$.

The \textit{spectral radius} of a square matrix $A$, written $\rho(A)$, is the largest absolute value of its eigenvalues.
\subsection{Tools}
Since we do not need the most general form of the Perron--Frobenius Theorem~\cite{Per1907, Fro1912}, we state only two of its corollaries.
\begin{lemma}[Perron's Theorem]
\label{theorem: perron}
For an $n$-by-$n$ matrix with positive entries, the spectral radius is an eigenvalue of multiplicity~$1$, such that its eigenvector has positive entries.
\end{lemma}
\begin{lemma}[Frobenius's Theorem]
\label{theorem: frobenius}
For an $n$-by-$n$ matrix with non-negative entries, the spectral radius is an eigenvalue such that one of its eigenvectors has non-negative entries.
\end{lemma}
Also, we need two finite-dimensional versions of the Hahn--Banach Theorem in \cref{section main}.
\begin{theorem}
\label{theorem 1st separation}
If the origin $o$ of $\mathbb R^d$ does not lie in the convex hull of points $x_1, \dots, x_n\in \mathbb R^d$, then there exists a vector $y\in \mathbb R^d$ such that $\langle y, x_i\rangle >0$ for all $i\in [n]$.
\end{theorem}
\begin{theorem}
\label{theorem 2st separation}
If the origin $o$ of $\mathbb R^d$ is not an interior point of the convex hull of points $x_1, \dots, x_n\in \mathbb R^d$, then there exists a vector $y\in \mathbb R^d$ such that $\langle y, x_i\rangle \ge 0$ for all $i\in [n]$.
\end{theorem}
\section{Proof of Rankin's Theorem}
\label{section proofs}

We state the two parts of Rankin's Theorem as separate theorems.
\begin{theorem}[Rankin]
	\label{theorem obtuse}
	If $\{v_1,\dots, v_n\}$ is a set of non-zero vectors in $\mathbb R^d$ such that the angle between any two of them is larger than $\frac{\pi}{2}$, then $n\le d+1$.
\end{theorem}
\begin{proof}
Suppose to the contrary that $n\ge d+2$. Let $G$ be the \emph{Gram matrix} of the vectors $v_1,\dots, v_n$, that is, $G=(\langle v_i, v_j\rangle)_{i,j=1}^n$. Choose a positive $\lambda$ such that $\lambda>\langle v_i, v_i\rangle$ for all $i\in [n]$, and set $H=\lambda I_n-G$. The hypothesis of the theorem implies that all entries of $H$ are positive. By \cref{theorem: perron}, the spectral radius $\rho(H)$ is the largest eigenvalue of $H$ of multiplicity one.

Obviously, $\rank G\le d$ because $G$ is the Gram matrix of $d$-dimensional vectors. Hence $0$ is an eigenvalue of $G$ of multiplicity at least two, that is, $\lambda$ is an eigenvalue of $H$ of multiplicity at least two. Since the Gram matrix $G$ is positive semidefinite, and thus, all its eigenvalues are non-negative, $\lambda$ must be the largest eigenvalue of $H$. Therefore, $\lambda=\rho(H)$, contradicting the fact that the multiplicity of the largest eigenvalue is one.
\end{proof}
\begin{theorem}[Rankin]
	\label{theorem non-acute}
	If $\{v_1,\dots, v_n\}$ is a set of non-zero vectors in $\mathbb R^d$ such that the angle between any two of them is at least $\frac{\pi}{2}$, then $n\le 2d$.
\end{theorem}
\begin{proof}
Suppose to the contrary that $n\ge 2d+1$. Without loss of generality we can assume that the vectors $v_1, \dots, v_n$ are of unit length. Let $G$ be the Gram matrix of the vectors $v_1,\dots, v_n$. Set $H=I_n-G$. The hypothesis of the theorem implies that all entries of $H$ are non-negative. By \cref{theorem: frobenius}, the spectral radius $\rho(H)$ is the largest eigenvalue of $H$. 

Obviously, $\rank G\le d$ because $G$ is the Gram matrix of $d$-dimensional vectors. Hence, $0$ is an eigenvalue of $G$ of multiplicity at least $n-d$, and thus, $1$ is an eigenvalue of $H$ of multiplicity at least $n-d$. Since the Gram matrix $G$ is positive semidefinite, that is, $0$ is its smallest eigenvalue, $1$ is the largest eigenvalue of $H$. It means that $\rho(H)= 1$. Let $\lambda_1, \dots, \lambda_d, 1,\dots,1$ be the eigenvalues of $H$, indexed in non-decreasing order. Thus we have
\[
0=\tr (H)=\lambda_1+\dots+\lambda_d+(n-d),\text{and thus, } |\lambda_1+\dots+\lambda_d|=n-d\ge d+1.
\]
However, the last inequality contradicts $|\lambda_i|\le 1=\rho(H)$ for all $i\in [d]$.
\end{proof}
\section{Proofs of Caratheodory's and Steinitz's Theorems}
\label{section main}
\begin{theorem}[Carath\'{e}odory's Theorem]
	\label{theorem caratheodory}
	If the origin $o\in \mathbb R^d$ lies in the convex hull of points $v_1,\dots, v_n\in \mathbb R^d$, then there is a set $J\subseteq [n]$ of size at most $d+1$ such that $o$ lies in the convex hull of $\{v_j\;:\; j\in J\}$.
\end{theorem}
\begin{proof}
We use induction on $n$. The base case, $n=d+1$, being trivial,
it is sufficient to show that if $n> d+1$ and $o\not \in \conv\left\{v_i:i\in [n]\setminus\{j\}\right\}$ for all $j\in[n]$, then $o\not \in \conv\left\{v_i:i\in [n]\right\}$.

By \cref{theorem 1st separation}, for all $j\in [n]$, there is a vector $y_j$ such that $\langle x_i, y_j\rangle>0$ for all $i\in [n]\setminus\{j\}$. Choose a positive $\lambda$ such that $\lambda+\langle v_i, y_i\rangle>0$ for all $i\in[n]$ and set $H=\lambda I_n+V^tY$, where $V=[v_1,\dots, v_n]$ and $Y=[y_1,\dots, y_n]$. By \cref{theorem: perron}, the spectral radius $\rho(H)$ is the largest eigenvalue of $H$ of multiplicity one. 

Obviously, $\rank V^tY\le d$ because $V$ and $Y$ are $n$-by-$d$ matrices. Hence, $0$ is an eigenvalue of $G$ of multiplicity at least two, that is, $\lambda$ is an eigenvalue of $H$ of multiplicity at least two. It follows that $\rho(H)\ne \lambda$, so $\rho(H)>\lambda$, because the spectral radius cannot be less than a positive eigenvalue. Consider the eigenvector $x$ of the eigenvalue $\rho(H)$. By \cref{theorem: perron}, its entries are positive. Therefore, we obtain
\[
\left(\lambda I_n+V^tY\right)x=\rho(H)x,
\mbox{ and thus, }
V^tYx=(\rho(H)-\lambda)x.
\]
Hence, $Yx$ is a $d$-dimensional vector such that $\langle v_i, Yx\rangle>0$ for all $i\in[n]$, because all entries of the vector $(\rho(H)-\lambda)x$ are positive. So the points $v_1,\dots, v_n$ lie in an open half-space bounded by a hyperplane passing through the origin. Thus, their convex hull does not contain the origin.
\end{proof}
\begin{theorem}[Steinitz's Theorem]
	\label{theorem steinitz}
	If the origin $o$ of $\mathbb R^d$ is an interior point of the convex hull of points $v_1,\dots, v_n\in \mathbb R^d$, then there is a set $J\subseteq [n]$ of size at most $2d$ such that the point $o$ is interior of convex hull of
	$\{v_j\;:\; j\in J\}$.
\end{theorem}
\begin{proof}
Again, we use induction on $n$. The base case, $n=2d$, being trivial,
it is sufficient to show that if $n>2d$ and, for all $j\in [n]$, the origin $o$ is not an interior point of $\conv\left\{v_i:i\in [n]\setminus\{j\}\right\}$, then $o$ is not an interior point of $\conv\left\{v_i:i\in [n]\right\}$.

By \cref{theorem 2st separation}, for all $j\in[n]$, there is a non-zero vector $y_j$ such that $\langle v_i, y_j\rangle\ge 0$ for all $i\in [n]\setminus\{j\}$. If $\langle v_i, y_i\rangle\ge 0$ for some $i\in[n]$, then the origin is not an interior point of the convex hull of $v_1, \dots, v_n$. So, without loss of generality, we may assume that $\langle v_i, y_i\rangle =-1$ for all $i\in [n]$. Set $H=I_n+V^tY$, where $V=[v_1,\dots, v_n]$ and $Y=[y_1,\dots, y_n]$. By \cref{theorem: frobenius}, the spectral radius $\rho(H)$ is an eigenvalue of $H$. 

Obviously, $\rank V^tY\le d$ because $V$ and $Y$ are $n$-by-$d$ matrices, and hence, $0$ is an eigenvalue of $V^tY$ of multiplicity at least $(n-d)$, that is, $1$ is an eigenvalue of $H$ of multiplicity at least $(n-d)$. Suppose that $\rho(H)= 1$. Let $\lambda_1, \dots, \lambda_d, 1,\dots,1$ be the eigenvalues of $H$, indexed in non-decreasing order. Hence we have
\[
0=\tr (H)=\lambda_1+\dots+\lambda_d+(n-d),
\mbox{ and thus, }
|\lambda_1+\dots+\lambda_d|=n-d\ge d+1.
\]
But the last inequality contradicts $|\lambda_i|\le 1=\rho(H)$ for all $i\in [d]$. Hence $\rho(H)>1$ because the spectral radius cannot be less than a positive eigenvalue. Consider the eigenvector $x$ of the eigenvalue $\rho(H)$. By \cref{theorem: frobenius}, its entries are non-negative entries. Therefore, we get
\[
\left(I_n+V^tY\right)x=\rho(H)x,
\mbox{ and thus, }V^tYx=(\rho(H)-1)x.
\]
Hence, $\langle v_i, Yx\rangle\ge 0$ for all $i\in [n]$. Moreover, $Yx$ is a non-zero vector, because among non-negative entries of $(\rho(H)-1)x$ there is at least one positive. So the points $v_1, \dots, v_n$ lie in a closed half-space bounded by a hyperplane passing through the origin, that is, the point $o$ is not interior of their convex hull.
\end{proof}

\appendix
\section{Proofs of Separtion Theorems}
\label{section proofs of separation theorems}

Theorems~\ref{theorem 1st separation} and \ref{theorem 2st separation} are fundamental in the study of convex sets. For completeness, we include their proofs.

\subsection{Proof of Theorem~\ref{theorem 1st separation}}
By compactness of the set $\conv\{v_1,\dots,v_n\}$, it has a point  $v\in \conv \{v_1,\dots,v_n\}$ closest to the origin. Suppose that $\langle v, v_i\rangle\le 0$. Write $v_\delta:=(1-\delta) v+\delta v_i$. Thus we have
\begin{align}
\|v_\delta\|^2=\|(1-\delta) v+\delta v_i\|^2&=(1-\delta)^2 \|v\|^2+2(1-\delta)\delta \langle v, v_i\rangle+\delta^2 \|v_i\|^2\\
\label{equation delta}
&\le \|v\|^2-2\delta \|v\|^2+\delta^2 (\|v\|^2+\|v_i\|^2).
\end{align}
Since the coefficient for $\delta$ in \eqref{equation delta} is negative, $\|v_\delta\|^2>\|v\|^2$ for sufficiently small positive $\delta$. Points $v$ and $v_i$ lie in $\conv\{v_1,\dots, v_n\}$, so $v_\delta\in \conv \{v_1,\dots,v_n\}$, contradicting the choice of $v$.
Therefore, $\langle v, v_i\rangle>0$ for all $i\in [n]$.
\subsection{Proof of Theorem~\ref{theorem 2st separation}}
Set $K=\conv\{v_1,\dots, v_n\}$. Denote the \emph{ray} emanating from a point $x$ and passing through a point $y$ by $[x,y)$.

We claim that there exists a ray $[o,t)$ such that $[o,t)\cap K=\{o\}$. Otherwise, one can choose $d+1$ points $t_1,\dots, t_{d+1}\in K$ such that the point $o$ is an interior point of $\conv\{t_1, \dots, t_{d+1}\}\subseteq K$.

Consider a sequence $o_1, \dots, o_n,\dots$ of points distinct from $o$ on the ray $[o,t)$ and converging to $o$. By \cref{theorem 1st separation}, there exists a sequence of unit vectors $y_1,\dots, y_n,\dots$ in $\mathbb R^d$ such that $\langle y_i, v_j-o_i\rangle>0$ for all $j\in [n]$ and all positive integer $i$. Since the unit sphere is compact, there exists a subsequence of unit vectors $y_{i_1},\dots,y_{i_n},\dots$ converging to a unit vector $y\in \mathbb R^d$. Therefore, we obtain $\langle y, v_j\rangle\ge \langle y, o\rangle=0$.

\section*{Acknowledgments}

The work was done while the second named author was an academic visitor at EPFL.

The authors are grateful to J\'{a}nos Pach and Imre B\'{a}r\'{a}ny for fruitful discussions.

\bibliographystyle{amsalpha}
\bibliography{biblio}

\end{document}